\documentclass[a4paper, 12pt]{amsart}
\pdfoutput=1
\usepackage{amssymb,amsmath,txfonts,mathrsfs,titletoc,appendix}
\usepackage{graphicx}
\usepackage{latexsym}
\usepackage[alphabetic]{amsrefs}
\usepackage{geometry}
\usepackage{fancyhdr}

 \newtheorem{thm}{Theorem}[section]
 
 \newtheorem{lem}[thm]{Lemma}
 
  \newtheorem{conj}[thm]{Conjecture}
 \newtheorem{rem}[thm]{Remark}
 \numberwithin{equation}{section}
\newtheorem{lem*}{Lemma}
\newtheorem{cor*}{Corollary}

\begin{document}

\title[On Realization of Tangent Cones]{On Realization of Tangent Cones of Homologically Area-minimizing Compact Singular Submanifolds}

\author{Yongsheng Zhang}
\email{yongsheng.chang@gmail.com}
\date{\today}

\begin{abstract}

     We show that every area-minimizing hypercone and every oriented Lawlor cone in \cite{Law}
      can be realized as a tangent cone at a point of some homologically area-minimizing singular {compact} submanifold.
      In particular this generalizes the result of N. Smale \cite{NS}.
\end{abstract}
 
 \keywords{Minimizing hypercone, Lawlor cone, tangent cone, 
 homologically area-minimizing {\em compact} submanifold, realization problem, mollification of calibrations} 
 \subjclass{Primary~28A75, Secondary~53C38}

\maketitle
\section{Introduction}\label{Section1}

     Let $C$ be a $k$-dimensional cone over link $L\subset S^{n-1}(1)$
     in an Euclidean space $(\mathbb R^{n},g_E)$.
               We call $C$ area-minimizing (mass-minimizing)
               if $C_1=C\bigcap \bold{B}^{n}(1)$ has least mass
               among all integral (normal) currents (see \cite{FF}) with boundary $L$. 
       We say that a $d$-closed compactly supported integral current in a Riemannian manifold is 
       homologically area-minimizing (mass-minimizing)
       if it has least mass in its homology class of integral (normal) currents.

        A well-known result of Federer (Theorem 5.4.3 in \cite{F},
        also see Theorem 35.1 and Remark 34.6 (2) in Simon \cite{LS})
        asserts that a tangent cone at a point of an area-minimizing rectifiable current
        is itself area-minimizing.
        This paper studies its converse realization question by {compact} submanifolds ($\star$):
\\{\ }

\begin{quote}
{\em Can any area-minimizing cone be realized as a tangent cone at a point
of some homologically area-minimizing  {\em compact} singular submanifold?}
\end{quote}
{\ }

              Through techniques of geometric analysis and Allard's regularity theorem, N. Smale found realizations for
              all strictly minimizing, strictly stable hypercones (see \cite{HS}) in \cite{NS}.
              They are first examples of codimension one homological area-minimizers with singularities.

              Very recently, different realizations of many area-minimizing cones,
              including all homogeneous minimizing hypercones (classified by Lawlor \cite{Law}, also see \cite{BL} and \cite{Z2})
              and special {Lagrangian} cones,
              by extending local calibration pairs were discovered in \cite{Z12}.
              
              However in general the answer to $(\star)$ is still far to be known. 
              In this paper, we focus on two important classes of mass-minimizing cones $-$ 
              minimizing hypercones
                    \footnote{
                    By \cite{F2} or \cite{FM2},
                    the area-minimality of a hypercone is equivalent to its mass-minimality.                    
                    So we say minimizing for short.}
              and oriented Lawlor cones.
              \\{\ }
              
              For hypercones,
              two long-term standing conjectures (or equivalent versions) raised
              by Simon, Hardt and Simon respectively are the followings.
\begin{conj}
Except trivial examples in low dimensions, all minimizing hypercones are strictly area-minimizing?
\end{conj}

\begin{conj}
Any non-trivial strictly area-minimizing hypecone is always strictly stable?
\end{conj}
                 Up to now it is unclear how far it is for a minimizing hypercone
                 to be strictly stable and strictly area-minimizing.
                 An important characterization of minimizing hyercones in \cite{HS}
                 is that each of them possesses a canonical singular ``calibration".
\\{\ }

                 By Lawlor cones we mean area-minimizing cones shown in \cite{Law}.
                 He studied when certain preferred bundle structure
                     (somehow analogous to that in \cite{HS} for hypercones,
                      nevertheless involving curvatures more heavily without the limitation to codimension one)
                      of some angular neighborhood of a minimal cone exists,
                  and successfully added quite a few interesting new oriented area-minimizing cones
                 (and non-orientable area-minimizing cones in the sense of modulo $2$ as well).
                 In the oriented case, such bundle structure naturally induces a
                 ``calibration" of the cone that is singular in a set of codimension one
                 and possibly also along the cone.
\\{\ }

                        By virtue of these peculiar calibrations of minimizing hypercones
                        and oriented Lawlor cones, we obtain realizations for them.
                        
                        \begin{thm}\label{thm2}
                        Every minimizing hypercone
                        can be realized to $(\star)$.
                        \end{thm}
                        
                        \begin{rem}
                        Our construction removes the requirements
                        of a minimizing hypercone's being
                        strictly stable and being strictly minimizing in \cite{NS}.
                        Hence the case of codimension one is completely settled.
                        \end{rem}

                       \begin{thm}\label{thm1}
                       Every oriented Lawlor cone can be realized to $(\star)$.
                       \end{thm}

                       \begin{rem}
                       This answers affirmatively to $(\star)$ for lots of area-minimizing cones of higher codimensions,
                       for instance, a minimal cone $C$ over a product of two or more spheres satisfying
                       (1) $\dim(C)>7$, or (2) $\dim(C)=7$ with none of the spheres being a circle
                       (cf. Theorem 5.1.1 in \cite{Law}).
                       These cones do not split.
                        Namely, they cannot be written as products of two or more area-minimizing cones of lower dimensions
                       (vs. N. Smale \cite{NS2}).
                       \end{rem}
                       

The paper is organized as follows.
In \S\ref{s2} our preferred model $S$ of construction is introduced.
By a monotonicity result of Allard, we get Lemma \ref{1}
which helps us transform the global realization question to a local problem around $S$ in \S\ref{s4}.
Thus, we only need to construct a smooth metric $\bar g$ on some neighborhood $\tilde U$ of $S$
such that $S$ is homologically area-minimizing in $\tilde U$.

We discuss the case of codimension one in \S\ref{s5}.
There are two steps.
              First suitably extend the canonical (local, singular and non-coflat) calibrations around $p_1$ and $p_2$ (see \S\ref{s2})
to a $C^1$ closed form $\Phi$ in a neighborhood $\tilde U$ of $S$.
               Then a smooth metric $\tilde g$ can be created to make $\Phi$ a $C^1$ calibration of $S$.
Hence we gain the homological area-minimality of $S$ in $\tilde U$.

        In \S\ref{s6} realizations of oriented Lawlor cones are constructed.
        The idea is roughly the same.
        However the calibration is discontinuous in a set of codimension one.
        So we consider its regularization through convolution for the desired local homological area-minimality of $S$.
        Although the approximating closed forms may have comass greater than one somewhere, 
        by the mildness of calibrations in \cite{Law} and {Lebesgue}'s bounded convergence theorem,
        the needed area-minimality can be achieved.
\\{\ }

         \section{Model of Construction}\label{s2}
                          Given a $k$-dimensional cone $C\subset \mathbb R^N$.
                          As in \cite{NS}, consider $\Sigma_C\triangleq (C\times \mathbb R)\bigcap S^{N}(1)$ in $ \mathbb R^{N+1}$.
                         Let $M$ be an embedded oriented  connected compact $k$-dimensional submanifold in
                         some $N$-dimensional oriented compact manifold $T$ with $[M]\neq [0]\in H_k(T;\mathbb Z)$.
                         Within smooth balls round a point of $M$ and a regular point of $\Sigma_C$ respectively
                         one can connect $T$ and $S^N(1)$, $M$ and $\Sigma_C$ simultaneously through one connected sum.
                         Denote by $X$ and $S$ the resulting manifold and submanifold (singular at two points $p_1$ and $p_2$).
                         Apparently $[S]\neq [0]\in H_k(X;\mathbb Z)$.
\\{\ }

       \section{Positive Lower Bound of Mass}\label{s3}
       
                         The lemma below will play a key role in \S\ref{s4}.
                         \begin{lem}\label{1}
                                       Let $g$ be a metric on a compact manifold $X$, 
                                       $W\Subset X$ an open domain where $\overline W$ forms a manifold with nonempty boundary $\partial \overline W$,
                                       and $\alpha$ a positive number.
                                       Then there exists $\beta=\beta_{\alpha,g|_{\overline W}}>0$ such that
                                       every rectifiable current $K$ in $W$
                                       with no boundary, vanishing generalized mean curvature vector field $\delta K$
                                       and at least one point in its support $\alpha$ away from $\partial \overline W$
                                       has mass greater than $\beta$.
                         \end{lem}

\begin{proof}
                        By Nash's embedding theorem \cite{Nash},
                        $(\overline W,g|_{\overline W})$ can be isometrically embedded through a map $f$ into
                        some Euclidean space $(\mathbb R^s, g_E)$.
                        Then $f_\#K$ is a rectifiable current of $f(\overline W)$.
                        Denote the induced varifold by $V_{f_\#K}$.
                        Since $K$ has no boundary in $W$ and $\delta K$ vanishes,
                        the norm of $\delta V_{f_\#K}$ in $\mathbb R^s$ is bounded from above a.e. by a constant $A$ depending upon $f$ only. 

                        Let $\overline{W_\alpha}=\{x\in W: \mathrm{dist}_g(x,\partial \overline W)\geq \alpha\}$.
                        Define $2\mu=\mathrm{dist}_{g_E}(f(\overline{W_\alpha}),f(\partial \overline W))$.
                        Note that the density of $V_{f_\#K}$ 
                        is a.e. at least one on the support $\mathrm{\bold{spt}}(f_\#K)$ of $f_\#K$.
                        Therefore there exists some point $p\in\bold{spt}(f_\#K)\bigcap  f(W)$
                        with  $\lambda\triangleq\mathrm{dist}_{g_E}(p,f(\partial \overline W))>\mu$ and density at least one.

                        By applying the following monotonicity result of Allard to $A,\ p,\ \mu$ and $U$ the open $\lambda$-ball centered at $p$,
                        we obtain our statement.
                        
                        \begin{thm}[\cite{Allard}]
                        Suppose $0\leq A<\infty$, $p\in \text{support of }\|V\|$, $V\in \bold{V}_m(U)$,
                        where $U$ is an open region of $\mathbb R^s$.
                        If $0<\mu<\text{dist}_{g_E}(p,\mathbb R^s-U)$ and 
                                        $$\|\delta V\|\bold B (p,r)\leq A\|V\|\bold B(p,r)\ \ \ \ \ whenever\ 0<r\leq \mu,$$
                        then $r^{-m}\|V\|\bold B(p,r)\exp Ar$ is nondecreasing in $r$ for $0<r\leq \mu$.
                        \end{thm}
\end{proof}
{\ }

\section{Reduction of $(\star)$ from Global to Local}\label{s4}
             The following theorem indicates that the essential difficulty of $(\star)$ comes from local. 
             Hence in \S\ref{s5} and \S\ref{s6} we make constructions on some neighborhood of $S$ only.
             
             \begin{thm}\label{fgtl}
                        Suppose $S$ is homologically area-minimizing in $(U,\bar g)$
                        where $U$ is an open neighborhood of $S$ and $\bar g$ is a smooth metric on $U$.
                        Then there exists a smooth metric $\hat g$ on the compact manifold $X$ such that
                        $S$ is homologically area-minimizing in $(X,\hat g)$.
             \end{thm}

\begin{proof}
             
             Take open neighborhoods $W$, $W'$ and $W''$ of $S$ so that $W''\Subset W'\Subset W\Subset U$
             and the closer of $W$ ($W'$ and $W''$ respectively) is a manifold with nonempty boundary.
             Extend $\bar g$ to a metric $\tilde g$ on $X$ with 
                     $$\tilde g|_W=\bar g|_W.$$
             Set $\alpha=\text{dist}_{\tilde g}(\partial \overline {W'},\partial \overline W)$.
             Let $\beta$ be the lower bound in Lemma \ref{1} for $\alpha$, domain $\overline {W'}^c$ and $\tilde g|_{\overline {W'}^c}$.
             Choose $\gamma=(t\beta^{-1}{\text{Vol}_{\tilde g}(S))^{-\frac{2}{k}}}<1$ for some large constant $t> 1$.
             Then construct $\hat g$ as follows.
                               \begin{equation}\label{G}
                               \hat g=
                               \begin{cases}
                               \gamma\tilde g& \text{on } W''\\
                               h\tilde g& \text{on } W''\sim W'\\
                               \tilde g& \text{on } X\sim W'
                               \end{cases}
                               \end{equation}
               where
                $h$ is a smooth function on $\overline{W'}\sim W''$, no less than $\gamma$ and equal to one near $\partial \overline{W'}$.
 \\{\ }
 
                          Now we show that $S$ is homologically area-minimizing
                           in $(X,\hat g)$.
\\{\ }

                          By the celebrated compactness result in Federer and Fleming \cite{FF}
                          there exists an area-minimizing current $T$ in $[S]$ with nonempty $\bold{spt}T$. 
\\{\ }

              Case One: $\bold{spt}T$ is not contained in $W$.
              According to our construction, $\bold{M}(S)=\frac{\beta}{t}<\beta<\bold{M}(T)$ by Lemma \ref{1}.
              Contradiction.
 \\{\ }
 
              Case Two: $\bold{spt}T\subset W$. 
              By assumption and \eqref{G} $S$ is homologically area-minimizing in $(W,\hat g|_W)$. 
              As a result, $S$ and $T$ share the same mass.
              Hence
              $S$ is homologically area-minimizing in $(X,\hat g)$.
  \end{proof}
                            \begin{rem}
                            $[S]\neq[0]\in H_k(X;\mathbb Z)$ is crucial in our proof.
                            \end{rem}
  {\ }
  

 
   \section{Realization of Minimizing Hypercones}\label{s5}
   
              Choose a metric $g$ for our model in \S\ref{s2} such that
              \\
                   (i). balls $\bold{B}^g_{p_i}(1)$ of radius one centered at $p_i$ are disjoint, and
              \\
                   (ii). local model $S\bigcap \bold B_{p_i}^g(1)$ in $(\bold B_{p_i}^g(1),g|_{\bold B_{p_i}^g(1)})$
                   is exactly $C_1$ in $\mathbf (\bold B^{N}(1),g_E|_{\bold B^{N}(1)})$.
                                    \begin{figure}[ht]
                                    \begin{center}
                                    \includegraphics[scale=0.3]{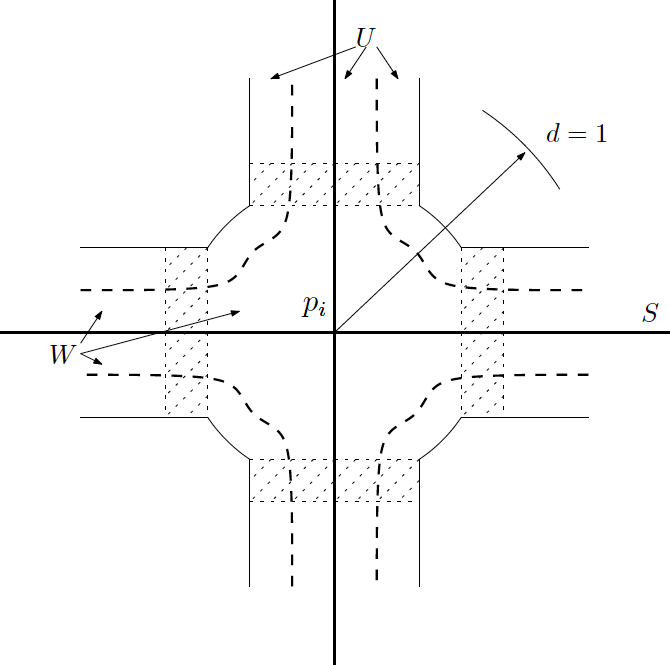}
                                    \end{center}
                                    \end{figure}
              {\ }\\{\ }\\
              Now take $U$ to be an open neighborhood of $S$ shown in the picture.
              \\{\ }
              
              Let us recall a beautiful result due to Hardt and Simon.
              
              \begin{thm}[Theorem 2.1 in \cite{HS}]
              Assume $C$ is an area-minimizing hypercone in $\mathbb R^{N}$.
              If $E$ is either one of the components $E_+$, $E_-$ of $\mathbb R^{N}\sim C$,
              then there is a unique oriented connected embedded real analytic minimizing hypersurface $H\subset E$
              with $H=\partial [[F]]$, $\overline F\subset \overline E$, $F$ open, the singular set of $H$ empty and the distance of $H$ and the origin equal to one.
              Moreover, $H$ has the property that for any $\xi\in E$ the ray $\{t\xi:t>0\}$ intersects $H$ in a single point.              
              \end{thm}

              Hence $E_\pm$ is foliated by $\Gamma_\pm=\{tH_\pm:t>0\}$.
              Let $X_\pm$ be the oriented unit normal vector of $\Gamma_\pm$ with limit $v_C$ (pointing into $E_+$) along $C\sim 0$,
              and $\phi_\pm$ the oriented volume form of $\Gamma_\pm$. 
              On $\mathbb R^{N+1}\sim 0$, define
                \begin{equation*}
                \phi=\begin{cases}\label{glueforms}
                \phi_+ &\text{ in }E_+\\
                \lim \phi_+(=\lim \phi_-)&\text{ in } C\sim 0\\
                \phi_- &\text{ in }E_-\\
                \end{cases}
                \end{equation*}
              According to \cite{HS}, outside some large ball, each $H_{\pm}$ is a graph of some $C^2$ function on $C$,
              so $\phi$ is $C^1$ along $C\sim 0$ and smooth elsewhere.
{\ }\\

              Our strategy is the following.
              
             \textbf{Step 1}: glue such forms around $p_1$ and $p_2$ to a form $\Phi$ in some neighborhood of $S$.
             
             \textbf{Step 2}: construct a smooth metric on the neighborhood so that $\Phi$ is a singular calibration of $S$.
             
             In this way a realization of a minimizing hypercone can be produced based upon \S\ref{s4}.
             \\{\ }

               Assume, for some $0<3R<1$, $\bold B_{p_i}(3R)\subset U$.
               Let $\bold r$ be the distance to the origin along $C$
                and $\Theta$ a small angular neighborhood over $C \bigcap \{1.4R<\bold r<2R\}$ shown in the figure.

                            \begin{figure}[ht]
                                    \begin{center}
                                    \includegraphics[scale=0.3]{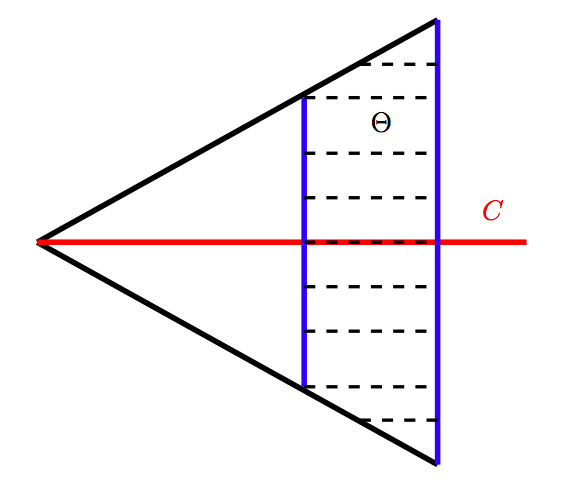}
                                    \end{center}
                            \end{figure}

              Set $\omega$ to be the unit volume form of the link $L$ of $C$ and $\psi=\bold r\omega$.
              Then $d\psi$ is the oriented unit $N$-dimensional form of $C\sim 0$.
              Since $div~ X_\pm=0$, one has (shrink $\Theta$ if necessary)
              $$\phi|_\Theta =[\pi^*d\psi]|_\Theta=[d(\pi^*\psi)]|_\Theta,$$
              where $\pi$ is the projection along $X_\pm$.
              On $\Theta$, let $\varpi$ be the projection to the nearest point on $C$
              and $\bold r=\bold r(\varpi(\cdot))$.
              Define
               $$
               \Phi=d[\tau(\bold r) (\pi^*\psi) +(1-\tau(\bold r))(\varpi^*\psi)],
               $$
               where $\tau$ is a decreasing smooth function from value one to zero on $[1.4R,2R]$ with the support of $d\tau$ contained in $[1.6R,1.7R]$.
               Note that $\Phi$ is the unit volume form of the cone in $\{1.4R<\bold r<2R\}\bigcap C$ . 
               \\{\ }
               
               For \textbf{Step 2}, we do some estimate on $\Phi$.
               Let $V$ be the parallel extension of $v_C$ along fibers of $\varpi$,
               $V^\perp$ the oriented unit $N$-vector perpendicular to $V$.
               Then on $\overline{E_+\bigcap \Theta}$
                \begin{equation}\label{long}
               \begin{split}
               &L_V\Phi\\
               =&L_Vd[\tau(\bold r) (\pi^*\psi) +(1-\tau(\bold r))(\varpi^*\psi)]\\
                             =&d[L_V(\tau(\bold r) (\pi^*\psi) +(1-\tau(\bold r))(\varpi^*\psi))]\\
                             =&d[\tau(\bold r) L_V(\pi^*\psi)]+d[(1-\tau(\bold r))L_V(\varpi^*\psi))]\\
                             =&(d\tau(\bold r))\wedge [i_V(d(\pi^*\psi))+d(i_V(\pi^*\psi))]+\tau(\bold r) d[L_V(\pi^*\psi)]\\
                             =&(d\tau(\bold r))\wedge [i_V\phi+\pi^*d(i_{\pi_*V}\psi)]+\tau(\bold r) [L_V\phi]
                            \\
               \end{split}
               \end{equation}
               Note that 
               \begin{equation}\label{W1}
               \varpi_*(V^\perp)=[1+O(\bold d^2_{g_E})]V^\perp|_C
               \end{equation}
                for the minimal cone $C$,
                 where $\bold d_{g_E}(\cdot)$ is the Euclidean distance to $C$.
                Consequently,
                \begin{equation}\label{W2}
                (L_VV^\perp)|_C=0.
                \end{equation}
               Therefore by \eqref{long} and \eqref{W2}
               \[
               (L_V[\Phi(V^\perp)])|_C=(L_V\Phi)|_C(V^\perp|_C).
               \]
               By the foliation structure, 
               it follows from \eqref{long} that
                \[
               (L_V\Phi)|_C=\tau(\bold r) [L_V\phi]|_C.
               \]
               Since $\phi$ is a calibration, we obtain
               \begin{equation}\label{+}
               (L_V[(\Phi(V^\perp)])|_C=\tau(\bold r) (L_V[(\phi(V^\perp)])|_C\leq 0.
               \end{equation}
               The same argument on $\overline{E_-\bigcap \Theta}$ produces
               \begin{equation}\label{-}
               (L_{-V}[(\Phi(V^\perp)])|_C=\tau(\bold r) (L_{-V}[(\phi(V^\perp)])|_C\leq 0.
               \end{equation}
               {\ }
               
               Hence, \eqref{+}, \eqref{-} and the compactness of $[1.4R,2R]$ imply that
               there exists a positive constant $K$
               such that in a sufficiently small neighborhood $\Xi$ of $C\bigcap \Theta$ in $\Theta$
               \begin{equation}\label{ineq}
               \Phi(V^\perp)\leq1+K\bold d_{g_E}^2.
               \end{equation}
               {\ }

               Now consider the {\em smooth} metric on $\Xi$
               \begin{equation}\label{g2}
               \hat g=(1+K\varrho(\bold r)\bold d_{g_E}^2)^{\frac{2}{N}}g_E,
               \end{equation}
               where $\varrho$ is a smooth increasing function with value zero on $[1.4R,1.5R]$ and value one on $[1.6R,2R]$.
               Set
              \begin{equation}\label{g3}
              \check g= \rho(\bold r)\hat g+(1-\rho(\bold r))(\|\varpi^*d\psi\|^*_{g_E})^{\frac{2}{N}}g_E,
              \end{equation}
               where $\rho$ is one on $[1.4R,1.8R]$, decreases to zero on $[1.8R,1.9R]$ and keeps value zero on $[1.9R,2R]$.
               On $[1.7R,2R]$, since $\Phi(V^\perp)=\|\varpi^*d\psi\|^*_{g_E}$,
                    \eqref{ineq} guarantees 
                           $$\check g\geq (\|\varpi^*d\psi\|^*_{g_E})^{\frac{2}{N}}g_E.$$               
               Therefore, on $1.4R\leq \bold r\leq 2R$,
                \[
                \Phi(V^\perp_{\check g})\leq1,
                \]
               where $V^\perp_{\check g}$ is the the oriented unit $N$-vector perpendicular to $V$ under ${\check g}$.            
               \begin{figure}[h]
                                     \begin{center}
                                     \includegraphics[scale=0.25]{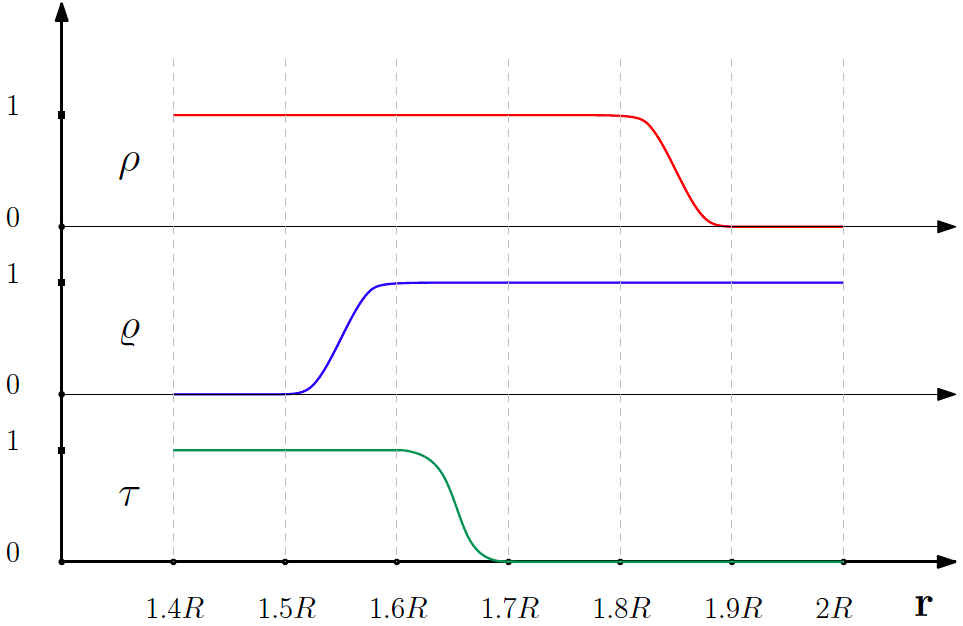}
                                     \end{center}
                                     \end{figure}

               By Lemmas {2.12} and {2.14} in Harvey and Lawson \cite{HL1}
               there exists a continuously varying $1$-dimensional plane field $\mathscr W$ transverse to $V^\perp_{\check g}$
               for $1.4R\leq \bold r\leq 2R$
               such that under the orthogonal combination $\tilde g= \check g|_{V^\perp}\oplus\tilde \alpha \check g|_{\mathscr W}$
               for some sufficiently large constant $\tilde \alpha$
               \[
               \|\Phi\|^*_{\tilde g}=\Phi(V^\perp_{\check g})\leq 1.
               \]
               However a vital flaw is that 
               $\tilde g$ may be NOT smooth.
              To conquer this, note that the angle between $V$ and $\mathscr W$
                   can be assumed strictly less than $\frac{\pi}{4}$ 
                   (the angle of $V$ and $\mathscr W$ being $0$ along ${C\bigcap\Xi}$)
                   in 
                   $\Xi$
                   on $1.4R\leq \bold r\leq 2R$.
              We define a {\em smooth} metric 
              \[
              \bar g=\check g|_{V^\perp}\oplus[1+ \varrho(\bold r)\rho(\bold r+0.1R) \sqrt{2}\tilde\alpha] \check g|_{V}
              \]
              on $\Xi$. (The shift term $0.1R$ is in fact not necessary.)
              Since 
              \begin{equation*}
                             \begin{split}
               &\text{on } [1.4R,1.6R],\ \ \ \ \  \|\Phi\|^*_{\bar g}\leq  \|\Phi\|^*_{\check g}=\|\phi\|^*_{\hat g}\leq \|\phi\|^*_{g_E}=1;\\
               &\text{on } [1.6R,1.7R],\ \ \ \ \ \|\Phi\|^*_{\bar g}\leq  \|\Phi\|^*_{\tilde g}\leq 1;\text{ and }\\
               &\text{on } [1.7R,2R],\ \ \ \ \ \ \ \ \|\Phi\|^*_{\bar g}\leq  \|\Phi\|^*_{\check g}=\Phi(V^\perp_{\check g})\leq 1,\\
                              \end{split}
               \end{equation*}     
we have
               \[
               \|\Phi\|^*_{\bar g} \leq 1.
               \]
               On $[1.4R,1.5R]$, $\Phi=\phi$ and $\bar g=g_E$.
               Meanwhile, on $[1.9R,2R]$, $\Phi=\varpi^*(d\psi)$ and $\bar g=\check g=(\|\varpi^*d\psi\|^*_{g_E})^{\frac{2}{N}}g_E$.
               {\ }\\{\ }
               
               It is apparent that this calibration pair of the $C^1$-calibration $\Phi$ and the smooth metric $\bar g$
                      can naturally extend on some neighborhood $\tilde U$ of $S$ in our model in \S\ref{s2}.
               According to Theorem 6.2 in \cite{F2} $S$ is homologically area-minimizing in $\tilde U$.
{\ }

\section{Realization of Oriented Lawlor Cones}\label{s6}
                Lawlor found many mass-minimizing cones in \cite{Law} by constructing particular
                calibrations discontinuous along boundary $\mathfrak B$ of some open angular neighborhood $\mathcal N$ for each of them.
                They are of form $\phi=d(f\tilde\psi)$ where $\tilde\psi$ is a smooth $(k-1)$-form on $\overline{\mathcal N}$ and
                where $f$ is at least $C^2$ along the cone and smooth elsewhere on $\mathcal N$,
                      Lipschitzian along $\mathfrak B$ with value zero on $(\overline{\mathcal N})^c$.
                Although $\phi$ is not continuous,
                through mollifications
                all oriented cones with such calibrations can be shown mass-minimizing.
                We will use the same idea.

                    First, one can similarly follow \textbf{Step 1} and \textbf{Step 2} in \S\ref{s5} with certain modifications.
                    Here most notations are taken directly from \S\ref{s5}.
                    
                    Recall $\psi=\bold r\omega$
                    where $\omega$ is the unit volume form of the link $L$ of an oriented Lawlor cone $C$.
              Then $d\psi$ is the oriented unit $k$-dimensional form of $C\sim 0$, and          
                                            $$\phi=d(f\cdot\varpi^*\psi)$$
                where $f(q)=\tilde f(\tan \theta(q))$ and $\theta(q)$ is the angle between $\overrightarrow {Oq}$ and $\overrightarrow{O(\varpi(q))}$.                
               Set $t=\tan\theta(q)=\frac{\bold{d}_{g_E}(q)}{\bold r(q)}$.
               According to \cite{Law}
               $\tilde f(t)=1-at^2-bt^3+\cdots$ near $t=0$.

                Define
               $$
               \Phi=d[\tau(\bold r) (f\cdot\varpi^*\psi) +(1-\tau(\bold r))(\varpi^*\psi)].
               $$

                For $q\in \mathcal N\sim C$, define $V_q=\frac{\overrightarrow{\varpi{(q)}q}}{|\overrightarrow{\varpi{(q)}q}|}$.
                Then we get a unit vector field $V$ on $\mathcal N\sim C$ whose limits on $C\sim 0$ give normal directions of $C\sim 0$.
                For $q\in \mathcal N$, denote by $F_q^\perp$ the oriented unit $k$-vector perpendicular to the fiber through $q$ and 
                it gives a $k$-vector field $F^\perp$ in $\mathcal N$.
                Since $L_V(\varpi^*\psi)=0$,
                 \begin{equation}\label{long2}
               \begin{split}
               &L_V\Phi\\
               =&L_Vd[\tau(\bold r) (f\varpi^*\psi) +(1-\tau(\bold r))(\varpi^*\psi)]\\
                             =&d[L_V(\tau(\bold r) (f\varpi^*\psi) +(1-\tau(\bold r))(\varpi^*\psi))]\\
                             =&d[f\tau(\bold r) L_V(\varpi^*\psi)]+d[L_V(f)\tau(\bold r)\varpi^*\psi]+d[(1-\tau(\bold r))L_V(\varpi^*\psi))]\\
                             =&d[L_V(f)\tau(\bold r)\varpi^*\psi]\\
               \end{split}
               \end{equation}
                              
               Let $\gamma(s)=\exp_p(s\nu)$ for $0\leq s<\epsilon$ where $\nu$ is a normal direction at a point $p$ of $C\sim 0$ and $\epsilon$ is small enough.
               So $\gamma'(s)=V_{\gamma(s)}$ for $0<s<\epsilon$ with $\lim _{s\rightarrow 0}V_{\gamma(s)}=\nu$.
               By Lemma 2.3.2 in \cite{Law},
               $$\lim_{s\rightarrow 0}(L_VF^\perp)_{\gamma(s)}=\left(\frac{d}{ds}|_{s=0}\det[I-sh^{\nu}_{ij}]^{-1}\right)F_{p}^\perp=0,$$
               where $h^{\nu}_{ij}$ is the second fundamental form at $p$ in normal direction $\nu$.
               Note that by \eqref{long2} $$\lim_{s\rightarrow 0}(L_V\Phi)_{\gamma(s)}$$ involves a normal direction.
               Therefore
               $$\lim_{s\rightarrow 0}(L_V[\Phi(F^\perp)])_{\gamma(s)}=0.$$
               {\ }
               
               Hence there exists a positive constant $K$
               such that in a sufficiently small neighborhood $\Xi$ of $C\bigcap \Theta$ in $\Theta$
               \begin{equation}\label{ineq2}
               \Phi(F^\perp)\leq1+K\bold d_{g_E}^2.
               \end{equation}
               
               Then following the procedures in \S\ref{s5} one can obtain
               a pair of $\Phi$ and $\bar g$ on some neighborhood $\tilde U$ of $S$, such that
               
               (1). $\bar g$ is a smooth metric,
               
               (2). the comass of $\Phi$ is no larger than $1$ where it is defined, and
               
               (3). $\Phi$ is the oriented volume form of the cone along $C\sim 0$.
               {\ }\\
               
            Take a smaller neighborhood $Y$ of $S$ where $Y\Subset \tilde U$ 
                and $(\overline Y,\bar g|_{\overline Y})$ forms a manifold with boundary.
                Isometrically embed $\overline Y$ into some Euclidean space $(\mathbb R^s, g_E)$ thru $F$.
        By the compactness of $F(\overline Y)$ 
             there is $\tau>0$ such that
             the exponential map restricted to the $\tau$-disk normal bundle $\mathfrak D$ over $F(Y)$ is a diffeomorphism.
        Denote by $\mathfrak N$ the image of $\mathfrak D$ and by $\pi$ the induced projection.
        Choose an open neighborhood $W\Subset Y$ of $S$.
        Let $\lambda=\mathrm{dist}_{g_E}(\partial \overline{F(Y)}, \partial{\overline {F(W)}})$.
        Then mollify $\pi^*((F^{-1})^*(\Phi))$
               with averaging radius $\epsilon<\epsilon_0=\frac{1}{2}\min\{\lambda,\tau\}$
               in the region $\{x\in \mathfrak N: \mathrm{dist}_{g_E}(x,\partial  \overline{\mathfrak N})\geq \epsilon_0\}$ of $\mathbb R^s$.
        Denote the generated smooth forms by $\tilde \Phi_\epsilon$ and
        set $\Phi_\epsilon=F^*(\tilde \Phi_\epsilon|_{F(W)})$.
        By the commutativity of the exterior differentiation and mollification in $\mathbb R^s$, it follows $d \Phi_\epsilon=0$.
\\{\ }

          Now we show that $S$ is homologically area-minimizing in $(\overline W,\bar g|_{\overline W})$.
          By \cite{FF}
                          there exists a minimizer $T=\overrightarrow T\cdot \|T\|$ in $[S]$. 
            Note that, except a measure $0$ set $\mathscr S$, $\bold{spt}T$ is a disjoint union of countably many $C^1$ submanifolds (see \cite{F}) and
            denote the bad set $(\bold{spt} T\sim \mathscr S)\bigcap \mathfrak B\sim 0$ by $\mathscr B$.
            Then $\mathscr B=\mathscr C\coprod\mathscr O$
            where $\mathscr C=\{x\in\mathscr B: \overrightarrow {T_x}\in \wedge^kT_x\mathfrak B \}$ and $\mathscr O=\mathscr B\sim \mathscr C$.
            The decomposition is unique up to a $\|T\|$-measure $0$ set.
            Obviously $\mathscr O$ is of $\|T\|$-measure $0$.
            Although $\Phi$ is not well defined along $\mathfrak B$,
            $\Phi_x(\overrightarrow T_x)$ makes sense on $\bold{spt}T\sim (\mathscr S\bigcup \mathscr O)$ with value $0$ on $\mathscr C$
            (due to the construction of $\phi$ in \cite{Law}).
             Also note that the uniformly bounded real-valued measurable function sequence $\Phi_\epsilon(\overrightarrow T)$
             converges to $\Phi(\overrightarrow T) $ pointwise on $\bold{spt}T\sim (\mathscr S\bigcup \mathscr O)$ (i.e., almost $\|T\|$-everywhere).
             Applying {Lebesgue}'s bounded convergence theorem we have
       
                \[
                \bold{M}(S)=\int_S\Phi=\lim_{\epsilon \downarrow 0}\int_S\Phi_\epsilon=\lim_{\epsilon \downarrow 0}\int\Phi_\epsilon(\overrightarrow T)d\|T\|
                =\int\Phi(\overrightarrow T)d\|T\|\leq \bold{M}(T).
                 \]   
                                    \begin{rem}
                                    $\Phi_\epsilon$ for $0<\epsilon<\epsilon_0$ may have comass greater than one under $\bar g$.
                                    \end{rem}
     
                        \begin{rem}\label{R}
                Similar argument shows that all Cheng's examples of homogeneous area-minimizing cones of codimension $2$ in \cite{Ch}
                        (e.g. minimal cones over
                        $U(7)/U(1)\times SU(2)^3$ in $\mathbb R^{42}$,
                        $Sp(n)\times Sp(3)/Sp(1)^3\times Sp(n-3)$ in $\mathbb R^{12n}$ for $n\geq 4$,
                        and $Sp(4)/Sp(1)^4$ in $\mathbb R^{27}$)
               can be realized as well.
                       \end{rem}
{\ }

 \section*{Acknowledgement}
 The author is deeply grateful to Professor H. Blaine Lawson, Jr. for his constant guidance and encouragement.
 He also wishes to thank Professor Frank Morgan for drawing our attention to N. Smale's work,
 Professor Robert Hardt and Professor Leon Simon for helpful comments on area-minimizing hypercones
 during the 2013 Midwest Geometry Conference and 2014 series lectures at Tsinghua,
 and the MSRI for its warm hospitality, travel fund and the financial support of the NSF under Grant No. 0932078 000 during his residence in the 2013 Fall.   
\\{\ }\\

{\ }

\end{document}